\documentclass[12pt, reqno]{amsart}
     \makeatletter
     \def\section{\@startsection{section}{1}%
     \z@{.7\linespacing\@plus\linespacing}{.5\linespacing}%
     {\bfseries
     \centering
     }}
     \def\@secnumfont{\bfseries}
     \makeatother
\usepackage{amssymb}
\usepackage[final]{hyperref}
\usepackage{tikz}
 
 \usepackage{wrapfig}

\usepackage{enumerate}

\newtheorem{theorem}{Theorem}[section]

\newtheorem{corollary}{Corollary}[theorem]
\theoremstyle{definition}
 
\newtheorem{remark}{Remark}[section]
 
\numberwithin{equation}{section}

\usepackage{amssymb}
\usepackage[final]{hyperref}
\usepackage{tikz}
 
 \usepackage{wrapfig}

\begin{document}

\title{An Inversion Formula for the Gaussian Radon Transform for Banach Spaces}

\author{Irina Holmes}
\address{Department of Mathematics \\
 Louisiana State University \\
Baton Rouge, LA 70803 \\
e-mail: \sl irina.c.holmes@gmail.com}
\thanks{Research of I. Holmes is supported by the US Department of Education GAANN grant P200A100080} 


\begin{abstract}
We prove a disintegration theorem for the Gaussian Radon transform $Gf$ on Banach spaces and use the Segal-Bargmann transform on abstract Wiener spaces to find a procedure to obtain $f$ from its Gaussian Radon transform $Gf$. 
\end{abstract}


\maketitle

\section{Introduction} \label{S:intro}

In this paper we work with an abstract Wiener space $(H, B, \mu)$, where $(H, \|\cdot\|)$ is a real separable Hilbert space, $B$ is the Banach space obtained by completing $H$ with respect to a measurable norm $|\cdot|$, and $\mu$ is Wiener measure on $B$. We recall here that for every $x^* \in B^*$, the restriction $x^*|_H$ is continuous with respect to the Hilbert norm $\|\cdot\|$ on $H$. Consequently, $B^*$ may be continuously embedded as a dense subspace of $H$ through the linear map:
	\begin{equation} \label{E:dual1}
	B^* \rightarrow H \text{; } x^* \mapsto h_{x^*}
	\end{equation}
where $h_{x^*}$ is, for every $x^* \in B^*$, the unique element of $H$ such that:
	\begin{equation} \label{E:dual2}
	\left< h_{x^*}, \cdot \right> = x^*|_H.
	\end{equation}
We will let $H_{B^*}$ denote the image of this map in $H$:
	\begin{equation}
	H_{B^*} = \left\{ h_{x^*} \in H : x^* \in B^* \right\}.
	\end{equation}
With this notation, Wiener measure $\mu$ is the Gaussian measure on $B$ given by:
	\begin{equation} \label{E:AWS}
	\int_B e^{i(x, x^*)}\,d\mu(x) = e^{-\frac{1}{2}\left\|h_{x^*}\right\|^2}
	\end{equation}
for all $x^* \in B^*$, where $(x, x^*)$ denotes the natural pairing $B$ - $B^*$ for all $x \in B$ and $x^* \in B^*$. For more details on abstract Wiener spaces, see \cite{Kuo1}.

The linear map $H_{B^*} \rightarrow L^2(B,\mu)$; $h_{x^*} \mapsto (\cdot, x^*)$ is $\|\cdot\|$-continuous and since $H_{B^*}$ is dense in $H$, this map has a unique extension to $H$, which we denote:
	\begin{equation} \label{E:pwin}
	I : H \rightarrow L^2(B, \mu) \text{; } h \mapsto Ih.
	\end{equation}
This map was first introduced by L. Gross in \cite{Gr} and is sometimes referred to as the ``Paley-Wiener map''. It is an isometry:
	\begin{equation} \label{E:iso}
	\left< Ih, Ik \right>_{L^2(B, \mu)} = \left< h, k \right>
	\end{equation}
for all $h, k \in H$, where $\left<\cdot,\cdot\right>$ denotes the inner product on $H$. Moreover, if $h \in H$, then any representative $\widetilde{h}$ of $Ih$ is a centered Gaussian random variable with variance $\|h\|^2$. Usually the map in \ref{E:pwin} is simply denoted $h \mapsto \widetilde{h}$, but some measure-theoretic technicalities arising in Section \ref{S:Disint} will require us to be a little careful about the true quotient space structure of $L^2$-spaces.

Although largely self-contained, this work is based on the results in \cite{HSen}, where an infinite-dimensional version of the Radon transform was developed for Banach spaces in the abstract Wiener space setting. In the absence of a useful version of Lebesgue measure on infinite-dimensional spaces, we constructed Gaussian measures on $B$ which are concentrated on hyperplanes and, more generally, on $B$-closures of closed affine subspaces of $H$. This result and some needed consequences are presented in Section \ref{S:GRT}. In Section \ref{S:Disint} we prove a disintegration of Wiener measure provided by these measures on closed affine subspaces and some ensuing corollaries. Among them, we establish the precise relationship between the Gaussian Radon transform $Gf$ and conditional expectation. These results are then used in conjunction with the Segal-Bargmann transform in Section \ref{S:Inversion} to prove an inversion theorem for $Gf$.

\section{The Gaussian Radon Transform} \label{S:GRT}

The following result was proved in Theorem 2.1 of \cite{HSen}:

\begin{theorem} \label{T:mump}
Let $(H, B, \mu)$ be an abstract Wiener space, $M_0$ be a closed subspace of $H$ and $\overline{M_0}$ denote the closure of $M_0$ in $B$. For every $p \in M_0^{\perp}$, there is a unique measure $\mu_{M_p}$ on $B$ such that:
	\begin{equation} \label{E:mumpchar}
	\int_B e^{i(x, x^*)}\,d\mu_{M_p}(x) = e^{i(p, x^*) - \frac{1}{2}\left\|P_{M_0}h_{x^*}\right\|^2}
	\end{equation}
for all $x^* \in B^*$, where $P_{M_0}$ denotes the orthogonal projection of $H$ onto $M_0$. Moreover, the measure $\mu_{M_p}$ is concentrated on $\overline{M_p} = p + \overline{M_0}$:
	\begin{equation} \label{E:mumpconc}
	\mu_{M_p}\left( \overline{M_p} \right) = 1.
	\end{equation}
\end{theorem}

We make a few observations about these measures.

\begin{enumerate}[(i)]

\item The measures defined above, though concentrated on different closed subspaces or translates of closed subspaces, are all defined on the same space $B$, which facilitates calculations involving more than one of these measures. In particular, if $M_0$ is a closed subspace of $H$ and $p \in M_0^{\perp}$, then:
	\begin{equation} \label{E:mumpmumorel}
	\int_B f(x) \,d\mu_{M_p}(x) = \int_B f(x + p) \,d\mu_{M_0}(x)
	\end{equation}
for all Borel functions $f$ on $B$ such that either side exists.

\item The space $(M_0, \overline{M_0}, \mu_{M_0})$ is itself an abstract Wiener space, where we are considering $\mu_{M_0}$ restricted to the Borel $\sigma$-algebra of $\overline{M_0}$. As a measure on $B$, $\mu_{M_0}$ is a centered degenerate Gaussian measure, since its covariance operator:
	\begin{equation} \label{E: mumocovop}
	q : B^* \times B^* \rightarrow \mathbb{R} \text{; } q\left( x^*, y^* \right) = \left< P_{M_0}x^*, P_{M_0}y^* \right> \text{, for all } x^*, y^* \in B^*
	\end{equation}
is not strictly positive definite. However, since $\overline{M_0}^*$ consists exactly of restrictions $x^*|_{\overline{M_0}}$ of elements $x^* \in B^*$, the restriction $q|_{\overline{M_0}^* \times \overline{M_0}^*}$ is strictly positive definite, and $\mu_{M_0}$ is non-degenerate on $\overline{M_0}$. It can also be easily shown that $M_0$ is indeed the Cameron-Martin space of $(\overline{M_0}, \mu_{M_0})$.

\item If $x^* \in B^*$ is $0$ on $M_0$, then $h_{x^*} \in M_0^{\perp}$ and $P_{M_0}h_{x^*} = 0$. The formula in \ref{E:mumpchar} then becomes:
	\begin{equation}
	\int_B e^{i(x, x^*)}\,d\mu_{M_p}(x) = e^{i(p, x^*)}
	\end{equation}
and, since $p \in M_0^{\perp} \subset H$, we have:
	\begin{equation}
	(p, x^*) = \left< p, h_{x^*} \right>.
	\end{equation}
Therefore:
	\begin{equation} \label{E:mumpae}
	(x, x^*) = \left<p, h_{x^*} \right> \text{ for } \mu_{M_p} \text{- almost all } x \in B
	\end{equation}
for all $x^* \in B^*$ such that $x^*|_{M_0} = 0$. Every $x^* \in B^*$ is then a (possibly degenerate) Gaussian random variable on $(B, \mu_{M_p})$, with mean $\left<p, h_{x^*}\right>$ and variance $\left\|P_{M_0}h_{x^*}\right\|^2$. 

\item If $M_0$ is a \textit{finite-dimensional} subspace of $H$, then $\mu_{M_0}$ is simply the standard Gaussian measure $\gamma_{M_0}$ on $\overline{M_0} = M_0$. To see this, note that from \ref{E:mumpchar}:
	\begin{equation}
	\int_{M_0} e^{i\left< h, P_{M_0}h_{x^*} \right>}\,d\mu_{M_0}(h) = e^{-\frac{1}{2}\left\|P_{M_0}h_{x^*}\right\|^2}
	\end{equation}
for all $x^* \in B^*$. Since $\left\{P_{M_0}h_{x^*} : x^* \in B^* \right\}$ is dense in $(M_0, \|\cdot\|)$, it follows by an application of the Lebesgue dominated convergence theorem that:
	\begin{equation}
	\int_{M_0} e^{i\left< h, k \right>}\,d\mu_{M_0}(h) = e^{-\frac{1}{2}\|k\|^2}
	\end{equation}
for all $k \in M_0$.

\item As in the classical case, the map:
	\begin{equation} \label{E:pwmapnotext}
	H_{B^*} \rightarrow L^2(B, \mu_{M_p}) \text{; } h_{x^*} \mapsto \left(\cdot, x^*\right)
	\end{equation}
is continuous with respect to the Hilbert norm $\|\cdot\|$ for every $p \in M_0^{\perp}$:
	\begin{eqnarray*}
	\left\| (\cdot, x^*) \right\|^2_{L^2(B, \mu_{M_p})} &=& \left< p, h_{x^*} \right>^2 + \left\| P_{M_0}h_{x^*} \right\|^2\\
	&\leq& \left( \|p\|^2 + 1 \right) \|h_{x^*}\|^2.
	\end{eqnarray*}
We denote the extension of this map to $H$ by:
	\begin{equation} \label{E:Imp}
	I_{M_p} : H \rightarrow L^2(B, \mu_{M_p}) \text{; } h \mapsto I_{M_p}h \text{, for all } p \in M_0^{\perp}.
	\end{equation}
If $h \in H$, then any representative $h_{M_p}$ of $I_{M_p}h$ is, with respect to $\mu_{M_p}$, Gaussian with mean $\left< p, h \right>$ and variance $\left\|P_{M_0}h\right\|^2$:
	\begin{equation} \label{E:hMpGauss}
	\int_B e^{ih_{M_p}} \,d\mu_{M_p} = e^{i\left< p, h \right> - \frac{1}{2}\left\| P_{M_0}h \right\|^2}.
	\end{equation}
However, unlike the classical case, $I_{M_p}$ is not an isometry:
	\begin{equation}
	\left< I_{M_p}h, I_{M_p}k \right>_{L^2(B, \mu_{M_p})} = \left< P_{M_0}h, P_{M_0}k \right>
	\end{equation}
for all $h, k \in H$.

\end{enumerate}

By a \textit{hyperplane} in a topological vector space we mean a translate of a closed subspace of codimension $1$. Suppose $P = au + u^{\perp}$ is a hyperplane in $H$, where $a \in \mathbb{R}$ and $u$ is a unit vector. As shown in \cite{HSen}, if $\left<\cdot, u\right>$ is continuous on $H$ with respect to the $B$-norm $|\cdot|$, then the $B$-closure $\overline{P}$ of $P$ in $B$ is a hyperplane in $B$ and, in fact, every hyperplane in $B$ is of this form. However if $\left<\cdot, u\right>$ is \textit{not} continuous with respect to $|\cdot|$, then $\overline{P} = B$ in $B$. The Gaussian Radon transform of a Borel function $f : B \rightarrow \mathbb{R}$ is defined on the set of all hyperplanes $P$ in $H$ as:
	\begin{equation} \label{E:gaussradon}
	Gf(P) \stackrel{\text{def}}{=} \int_B f \,d\mu_P
	\end{equation}
where $\mu_P$ is the Gaussian measure on $B$ concentrated on $\overline{P}$, as described in Theorem \ref{T:mump}.

\section{A Disintegration of Wiener Measure} \label{S:Disint}

\begin{theorem} \label{T:Disint}
Let $(H, B, \mu)$ be an abstract Wiener space and $Q_0$ be a closed subspace of $H$ with finite codimension. Then the map:
	\begin{equation} \label{E:GQ0}
	 Q_0^{\perp} \ni p \mapsto G_{Q_0}f(p) \stackrel{\text{def}}{=} \int_B f \,d\mu_{Q_p} 
	\end{equation}
is Borel measurable on $Q_0^{\perp}$ for all non-negative Borel functions $f$ on $B$. Moreover:
	\begin{equation} \label{E:disintformula}
	\int_B f\,d\mu = \int_{Q_0^{\perp}} \left( \int_B f \,d\mu_{Q_p} \right) \,d\gamma_{Q_0^{\perp}}(p)
	\end{equation}
for all Borel functions $f : B \rightarrow \mathbb{C}$ for which the left side exists, where $\gamma_{Q_0^{\perp}}$ denotes standard Gaussian measure on $Q_0^{\perp}$.
\end{theorem}

\begin{proof}
Suppose first that $f$ is a non-negative Borel function on $B$. Observe that:
	\begin{equation}
	G_{Q_0}f(p) = \int_B f(x + p) \,d\mu_{Q_0}(x)
	\end{equation}
for all $p \in Q_0^{\perp}$, so measurability of the map in \ref{E:GQ0} follows from Fubini's theorem.

To prove \ref{E:disintformula}, it suffices to show that the characteristic functional of the Borel probability measure $\mu'$ on $B$ given by:
	\begin{equation}
	\int_B g \,d\mu' \stackrel{\text{def}}{=} \int_{Q_0^{\perp}} \int_B g(x) \,d\mu_{Q_p}(x) \,d\gamma_{Q_0^{\perp}}(p)
	\end{equation}
for all bounded Borel functions $g$ on $B$, coincides with that of Wiener measure $\mu$. To see this, note that for all $x^* \in B^*$:
	\begin{eqnarray*}
	\int_B e^{i(x, x^*)}\,d\mu'(x) &=& \int_{Q_0^{\perp}} \int_B e^{i(x, x^*)}\,d\mu_{Q_p}(x) \,d\gamma_{Q_0^{\perp}}(p)\\
	&=& e^{-\frac{1}{2}\left\| P_{Q_0} h_{x^*} \right\|^2} \int_{Q_0^{\perp}} e^{i(p, x^*)}\,d\gamma_{Q_0^{\perp}}\\
	&=& e^{-\frac{1}{2}\left\| P_{Q_0} h_{x^*} \right\|^2} \int_{Q_0^{\perp}} e^{i\left< p, P_{Q_0^{\perp}}h_{x^*} \right>} \,d\gamma_{Q_0^{\perp}}(p)\\
	&=& e^{-\frac{1}{2}\left\| P_{Q_0} h_{x^*} \right\|^2} e^{-\frac{1}{2}\left\| P_{Q_0^{\perp}} h_{x^*} \right\|^2}\\
	&=& e^{-\frac{1}{2}\left\| h_{x^*} \right\|^2}
	\end{eqnarray*}
which proves our claim.
\end{proof}

Next, we explore some of the consequences of this result. First, suppose $f$ is a Borel function on $B$ such that $\|f\|_{L^r(B, \mu)} < \infty$ for some $1 \leq r < \infty$. Using $|f|^r$ in place of $f$ in \ref{E:disintformula}, we have:
	\begin{equation}
	\label{E:lr}
	|\!|f|\!|^r_{L^r(B, \mu)} = \int_{Q_0^{\perp}} |\!|f|\!|^r_{L^r(B, \mu_{Q_p})} \,d\gamma_{Q_0^{\perp}}(p) < \infty.
	\end{equation}
Consequently, the function $p \mapsto |\!|f|\!|_{L^r(B, \mu_{Q_p})}$ is $\gamma_{Q_0^{\perp}}$-almost everywhere finite on $Q_0^{\perp}$. Moreover, $G_{Q_0}f \in L^r(Q_0^{\perp}, \gamma_{Q_0^{\perp}})$:
	\begin{eqnarray*}
	\int_{Q_0^{\perp}} \left| G_{Q_0}f(p) \right|^r \,d\gamma_{Q_0^{\perp}}(p) &=& \int_{Q_0^{\perp}} \left| \int_B f \,d\mu_{Q_p} \right|^r \,d\gamma_{Q_0^{\perp}}(p) \\
	&\leq& \int_{Q_0^{\perp}} \left( \int_B |f|^r \,d\mu_{Q_p} \right) \,d\gamma_{Q_0^{\perp}}(p)\\
	&=& \int_B |f|^r \,d\mu < \infty.
	\end{eqnarray*}
	
Now for a fixed $h \in H$, let $\widetilde{h}$ be any representative of $Ih \in L^2(B, \mu)$ and for every $p \in Q_0^{\perp}$ let $h_{Q_p}$ be any representative of $I_{Q_p}h \in L^2(B, \mu_{Q_p})$. Let $\{y_{x^*_k}\}_{k \geq 1}$ be a sequence in $H_{B^*}$ converging to $h$ in $H$. Then:
	\begin{equation}
	\label{E:htilde}
	|\!|\widetilde{h} - x_k^*|\!|_{L^2(B, \mu)} \xrightarrow{k \rightarrow \infty} 0
	\end{equation}
and:
	\begin{equation}
	\label{E:hqp}
	|\!|h_{Q_p} - x^*_k|\!|_{L^2(B, \mu_{Q_p})} \xrightarrow{k \rightarrow \infty} 0 
	\end{equation}
for all $p \in Q_0^{\perp}$. From the disintegration formula in \ref{E:disintformula}:
	\begin{equation}
	|\!|\widetilde{h} - x^*_k|\!|^2_{L^2(B, \mu)} = \int_{Q_0^{\perp}} |\!|\widetilde{h} - x^*_k|\!|^2_{L^2(B, \mu_{Q_p})} \,d\gamma_{Q_0^{\perp}}(p),
	\end{equation}
so from \ref{E:htilde}:
	\begin{equation}
	\label{E:L2gamman}
	|\!|\widetilde{h} - x^*_k|\!|_{L^2(B, \mu_{Q_p})} \xrightarrow{k \rightarrow \infty} 0 \text{ in } L^2(Q_0^{\perp}, \gamma_{Q_0^{\perp}}).
	\end{equation}
But from \ref{E:hqp}:
	\begin{equation}
	\label{E:pwlimit}
	|\!|\widetilde{h} - x^*_k|\!|_{L^2(B, \mu_{Q_p})} \xrightarrow{k \rightarrow \infty} |\!|\widetilde{h} - h_{Q_p}|\!|_{L^2(B, \mu_{Q_p})}
	\end{equation}
for $\gamma_{Q_0^{\perp}}$-almost all $p \in Q_0^{\perp}$. Since pointwise almost everywhere limits and mean square limits coincide, \ref{E:L2gamman} and \ref{E:pwlimit} give us:
	\begin{equation}
	|\!|\widetilde{h} - h_{Q_p}|\!|_{L^2(B, \mu_{Q_p})} = 0 \text{ for } \gamma_{Q_0^{\perp}} \text{-almost all } p \in Q_0^{\perp}.
	\end{equation}
We make the remark that, in light of these results and some calculations needed in Section \ref{S:Inversion}, we would have liked to say something like ``$Ih = I_{Q_p}h$ for almost all $p$''. However, $L^2(B, \mu)$ and $L^2(B, \mu_{Q_p})$ are two completely different spaces whose elements are functions defined almost everywhere with respect to different measures. So for the sake of accuracy we summarize these findings as follows:

\begin{corollary}
\label{C:Lrhtilde}
Let $(H, B, \mu)$ be an abstract Wiener space and $Q_0$ be a closed subspace of finite codimension in $H$.
	\begin{enumerate}[(i). ]
	\item Let $[g] \in L^r(B, \mu)$ for some $1 \leq r < \infty$ and $f$ be any representative of $[g]$. Then: 
		\begin{equation}
		|\!|f|\!|_{L^r(B, \mu_{Q_p})} < \infty \text{ for almost all } p \in Q_0^{\perp}
		\end{equation}
	and 
		\begin{equation}
		\label{E:GqfinLr}
		G_{Q_0}f \in L^r(Q_0^{\perp}, \gamma_{Q_0^{\perp}}).
		\end{equation}
	In particular, if $f = 0$ $\mu$-almost everywhere for some measurable function $f$ on $B$, then $G_{Q_0}f(p) = 0$ for $\gamma_{Q_0^{\perp}}$-almost all $p \in Q_0^{\perp}$.
	\item Let $h \in H$ and $\widetilde{h}$ be any representative of $Ih$. Then $\widetilde{h}$ is a representative of $I_{Q_p}h$ for $\gamma_{Q_0^{\perp}}$-almost all $p \in Q_0^{\perp}$.
	\end{enumerate}
\end{corollary}

Another consequence of the disintegration theorem we explore is the relationship between $G_{Q_0}f$ and the conditional expectations of $f$.

\begin{corollary} \label{C:disint3}
Let $(H, B, \mu)$ be an abstract Wiener space, $f \in L^2(B, \mu)$ and $Q_0$ be a closed subspace of finite codimension in $H$. Let $\{u_1, \ldots, u_n\}$ be an orthonormal basis for $Q_0^{\perp}$ and for every $1 \leq k \leq n$ let $\widetilde{u_k}$ be any representative of $Iu_k$. Then:
	\begin{equation} \label{E:condexp}
	G_{Q_0}f\left( \widetilde{u_1}u_1 + \ldots + \widetilde{u_n}u_n \right) = \mathbb{E}\left[ f | \widetilde{u_1}, \ldots, \widetilde{u_n} \right]
	\end{equation}
\end{corollary}

\noindent Note that if $u_k = h_{x_k^*}$ for some $x_k^* \in B^*$, then $\widetilde{u_k}$ may simply be replaced above by $x^*_k$, as there are no tricky convergence issues.

\begin{proof}
Recall that for $p \in Q_0^{\perp}$ and $1 \leq k \leq n$, any representative of $I_{Q_p}u_k$ is, with respect to $\mu_{Q_p}$, Gaussian with mean $\left<p, u_k\right>$ and variance $\left\| P_{Q_0} u_k \right\|^2 = 0$, therefore it is $\mu_{Q_p}$-almost everywhere equal to $\left< p, u_k \right>$. From Corollary \ref{C:Lrhtilde}:
	\begin{equation} \label{E:cor3eq1}
	\widetilde{u_k}(x) = \left<p, u_k\right> \text{ for } \mu_{Q_p}\text{-almost all } x \in B \text{, for } \gamma_{Q_0^{\perp}} \text{-almost all } p \in Q_0^{\perp}.
	\end{equation}
Let $g\left(\widetilde{u_1}, \ldots, \widetilde{u_n}\right)$ be an element of $L^2\left(B, \sigma\left(\widetilde{u_1}, \ldots, \widetilde{u_n}\right)\right)$, where $g : \mathbb{R}^n \rightarrow \mathbb{R}$ is Borel measurable. By \ref{E:disintformula} and \ref{E:cor3eq1}:
	\begin{eqnarray*}
	\int_B g\left(\widetilde{u_1}, \ldots, \widetilde{u_n}\right) f \,d\mu &=& \int_{Q_0^{\perp}} \left( \int_B g\left(\widetilde{u_1}, \ldots, \widetilde{u_n}\right) f \,d\mu_{Q_p} \right) \,d\gamma_{Q_0^{\perp}}(p)\\
	&=& \int_{Q_0^{\perp}}g\left(\left<u_1, p\right>, \ldots, \left<u_n, p\right>\right)G_{Q_0}f(p) \,d\gamma_{Q_0^{\perp}}(p)\\
	&=& \int_B g\left(\widetilde{u_1}, \ldots, \widetilde{u_n}\right) G_{Q_0}f\left( \widetilde{u_1}u_1 + \ldots + \widetilde{u_n}u_n\right)\,d\mu
	\end{eqnarray*}
where the last equality follows from the fact that the distribution measure of the map:
	\begin{equation} \label{E:proj}
	\widetilde{P_{Q_0^{\perp}}} : B \rightarrow Q_0^{\perp} \text{; } \widetilde{P_{Q_0^{\perp}}}(x) \stackrel{\text{def}}{=}  \widetilde{u_1}(x)u_1 + \ldots + \widetilde{u_n}(x)u_n
	\end{equation}
is exactly standard Gaussian measure on $Q_0^{\perp}$.
\end{proof}

\begin{remark}
The author is grateful to the referee for pointing out that the formula in \ref{E:condexp} is equivalent to the elegant formulation:
	\begin{equation}
	G_{Q_0}f \circ \widetilde{P_{Q_0^{\perp}}} = \mathbb{E} \left[ f | \widetilde{P_{Q_0^{\perp}}} \right].
	\end{equation}
\end{remark}

\begin{remark}
This result tells us that for all unit vectors $u \in H$ the conditional expectation of $f$ given $\widetilde{u}$ is exactly the Gaussian Radon transform:
	\begin{equation}
	Gf\left( pu + u^{\perp} \right) = \mathbb{E}\left[ f | \widetilde{u} = p \right]
	\end{equation}
for almost all $p \in \mathbb{R}$, where $Gf$ is defined as in \ref{E:gaussradon}. An inversion formula for the Gaussian Radon transform will then provide a way to recover $f$ from its conditional expectations.
\end{remark}

\section{The Inversion Formula}
\label{S:Inversion}

We obtain an inversion formula using the Segal-Bargmann transform for abstract Wiener spaces. This transform is well-known for finite-dimensional real Hilbert spaces: for any $f \in L^2(\mathbb{R}^n, \gamma_n)$, where $\gamma_n$ denotes standard Gaussian measure on $\mathbb{R}^n$, the Segal-Bargmann transform of $f$ is the function $Sf : \mathbb{C}^n \rightarrow \mathbb{C}$ defined for all $\alpha = (\alpha_1, \ldots, \alpha_n) \in \mathbb{C}^n$ by:
	\begin{equation}
	\label{E:SBRn}
	(Sf)(\alpha) = e^{-\frac{1}{2}\alpha\cdot \alpha} \int_{\mathbb{R}^n} e^{\alpha \cdot x} f(x) \,d\gamma_n(x)
	\end{equation}
where $\alpha \cdot \beta = \alpha_1\beta_1 + \ldots + \alpha_n\beta_n$ for all $\alpha, \beta \in \mathbb{C}^n$.  

In the case of an abstract Wiener space $(H, B, \mu)$, we work with the complexification $H_{\mathbb{C}} = H \oplus iH$ of $H$. Let $\mathcal{F}$ denote the collection of all finite-dimensional subspaces of $H_{\mathbb{C}}$ and for every $F \in \mathcal{F}$ let $\lambda_F$ denote the measure on $F$ given by:
	\begin{equation}
	\,d\lambda_F(z) = \frac{1}{\pi^n}e^{-\|z\|^2}\,dz
	\end{equation}
where $n = $ dim$(F)$ and $dz$ is Lebesgue measure. The Segal-Bargmann space over $H_{\mathbb{C}}$, denoted $\mathcal{H}L^2(H_{\mathbb{C}})$, is defined as the space of all holomorphic functions $g$ on $H_{\mathbb{C}}$ such that:
	\begin{equation}
	\sup_{F \in \mathcal{F}} \int_F |g(z)|^2 \,d\lambda_F(z) < \infty.
	\end{equation}
For every $g \in \mathcal{H}L^2(H_{\mathbb{C}})$ we let $\|g\|_{SB}$ denote the square root of the above supremum. This is a complete inner-product norm and $\mathcal{H}L^2(H_{\mathbb{C}})$ is a complex Hilbert space. For every $f \in L^2(B, \mu)$ we define the Segal-Bargmann transform $S_Bf: H_{\mathbb{C}} \rightarrow \mathbb{C}$ of $f$ as:
	\begin{equation}
	\left(S_Bf\right)(z) = e^{-\frac{1}{2}(z, z)} \int_B e^{\widetilde{z}(x)} f(x) \,d\mu(x)
	\end{equation}
where $(z, w)$ denotes the complex \textit{bilinear} extension of the inner product on $H$ for all $z, w \in H_{\mathbb{C}}$ and $\widetilde{z} = \widetilde{h} + i\widetilde{k}$ for all $z = h + ik \in H_{\mathbb{C}}$ with $h, k \in H$ and $\widetilde{h}$, $\widetilde{k}$ any representatives of $Ih$, $Ik$ respectively. The map $S_B : L^2(B, \mu) \rightarrow \mathcal{H}L^2(H_{\mathbb{C}})$ is unitary. For more details and proofs of these results, see \cite{Se1}, \cite{DrHall}, \cite{Hall} or \cite{DriverGordina}.

\begin{theorem} \label{T:inversion}
Let $(H, B, \mu)$ be an abstract Wiener space, $f \in L^2(B, \mu)$ and $Q_0$ be a closed subspace of $H$ of finite codimension. Let $\{u_1, \ldots, u_n\}$ be an orthonormal basis for $Q_0^{\perp}$. Then for all $\alpha = (\alpha_1, \ldots, \alpha_n) \in \mathbb{C}^n$:
	\begin{equation} \label{E:inversion}
	\left(S_{Q_0^{\perp}} \left( G_{Q_0}f \right)\right)\left(\alpha_{Q_0^{\perp}}\right) = \left( S_Bf \right)(\alpha_{Q_0^{\perp}})
	\end{equation}
where $\alpha_{Q_0^{\perp}} = \alpha_1 u_1 + \ldots + \alpha_n u_n \in \left( Q_0^{\perp} \right)_{\mathbb{C}}$ and $S_{Q_0^{\perp}}$ and $S_B$ are the Segal-Bargmann transforms on $L^2(Q_0^{\perp}, \gamma_{Q_0^{\perp}})$ and $L^2(B, \mu)$, respectively.
\end{theorem}

\begin{proof}
From \ref{E:cor3eq1}:
	\begin{equation} \label{E:inveq1}
	\widetilde{\alpha_{Q_0^{\perp}}}(x) = \left( \alpha_{Q_0^{\perp}}, p \right) \text{ for } \mu_{Q_p}\text{-almost all } x \in B \text{, for } \gamma_{Q_0^{\perp}}\text{-almost all } p \in Q_0^{\perp}.
	\end{equation}
By Corollary \ref{C:Lrhtilde}, $G_{Q_0}f \in L^2(Q_0^{\perp}, \mu_{Q_0^{\perp}})$, so we may consider its Segal-Bargmann transform. From \ref{E:inveq1} and \ref{E:disintformula}:
	\begin{eqnarray*}
	\left(S_{Q_0^{\perp}}\left( G_{Q_0}f \right)\right)\left(\alpha_{Q_0^{\perp}}\right) &=& e^{-\frac{1}{2} (\alpha_{Q_0^{\perp}}, \alpha_{Q_0^{\perp}})} \int_{Q_0^{\perp}} e^{( \alpha_{Q_0^{\perp}}, p )} G_{Q_0}f(p) \,d\gamma_{Q_0^{\perp}}(p)\\
	&=& e^{-\frac{1}{2}( \alpha_{Q_0^{\perp}}, \alpha_{Q_0^{\perp}})}\int_{Q_0^{\perp}}\int_B e^{\widetilde{\alpha_{Q_0^{\perp}}}(x)} f(x) \,d\mu_{Q_p}(x) \,d\gamma_{Q_0^{\perp}}(p)\\
	&=& e^{-\frac{1}{2}(\alpha_{Q_0^{\perp}}, \alpha_{Q_0^{\perp}})}\int_B e^{\widetilde{\alpha_{Q_0^{\perp}}}(x)} f(x) \,d\mu(x)\\
	&=& \left(S_Bf\right)(\alpha_{Q_0^{\perp}}).
	\end{eqnarray*}
\end{proof}

As the referee pointed out, \ref{E:inversion} is equivalent to:
	\begin{equation}
	S_{Q_0^{\perp}}\left( G_{Q_0}f \right) = \left.\left(S_Bf\right)\right|_{(Q_0^{\perp})_{\mathbb{C}}}.
	\end{equation}
Now recall that any hyperplane in $H$ can be uniquely expressed as:
	\begin{equation}
	P(t, u) = tu + u^{\perp}
	\end{equation}
for $t > 0$ and $u$ a unit vector in $H$. The formula in \ref{E:inversion} yields that for all $h = \|h\|u$, where $u \in H$ is a unit vector, we have:
	\begin{equation}
	S\left[ Gf\left(P(\cdot, u)\right) \right]\left( \|h\| \right) = S_Bf(h)
	\end{equation}
where $S$ is the Segal-Bargmann transform on $\mathbb{R}$ and $Gf$ is the Gaussian Radon transform. Consequently, if we know $Gf(P)$ for all hyperplanes $P$ in $H$, then we know $S_Bf|_H$. Taking the holomorphic extension to $H_{\mathbb{C}}$, we know $S_Bf$ and can then obtain $f$ from the inverse Segal-Bargmann transform.

\section*{Acknowledgment}
The author acknowledges support from the US Department of Education GAANN grant P200A100080 and 
is deeply grateful to her adviser, Prof. Ambar Sengupta for his guidance.  This work is part of a project covered by grant NSA grant H98230-13-1-0210. The author is also grateful to the referee for an insightful and careful analysis of this paper.

\end{document}